\documentclass[11pt]{amsart}
\usepackage[
  top=3cm,
  bottom=2.5cm,
  left=2.5cm,
  right=2.5cm
]{geometry}
\usepackage[T1]{fontenc}
\usepackage{lmodern}
\usepackage{microtype}
\usepackage{amsmath,amssymb,mathtools,color}
\usepackage{booktabs}
\usepackage{enumitem}
\usepackage[colorlinks=true,linkcolor=blue,citecolor=blue,urlcolor=blue]{hyperref}
\usepackage[nameinlink,capitalise]{cleveref}
\newtheorem{theorem}{Theorem}[section]
\newtheorem{proposition}[theorem]{Proposition}
\newtheorem{lemma}[theorem]{Lemma}
\newtheorem{corollary}[theorem]{Corollary}
\newtheorem{conjecture}[theorem]{Conjecture}
\theoremstyle{definition}
\newtheorem{definition}[theorem]{Definition}
\theoremstyle{remark}
\newtheorem{remark}[theorem]{Remark}

\newcommand{\R}{\mathbb R}

\newcommand{\cO}{\mathcal O}
\newcommand{\Kah}{\operatorname{Kah}}
\newcommand{\Psef}{\operatorname{Psef}}

\title[Uniruledness and the sign of total scalar curvature]{Uniruledness and the sign of total scalar curvature}
\author{Zehao Sha \and Jian Wang}
\date{}

\begin{document}

\begin{abstract}
For every integer $n\ge3$, we construct a smooth projective manifold $X$ of complex dimension $n$ whose canonical bundle is not pseudoeffective, or equivalently, which is uniruled, but every K\"ahler metric has negative total scalar curvature. In particular, $X$ admits no K\"ahler metric of positive scalar curvature, while it admits a Riemannian metric of positive scalar curvature. Thus the equivalence between uniruledness and the existence of a K\"ahler metric with positive scalar curvature holds in complex dimensions one and two, but fails in higher dimensions.
\end{abstract}
\maketitle

\section{Introduction}

Let $(X,\omega)$ be a compact K\"ahler manifold of complex dimension $n$ and let $\alpha=[\omega] \in \Kah(X)$ be its K\"ahler class. The relationship between uniruledness and the existence of positive scalar curvature K\"ahler metrics connects the birational invariant of compact K\"ahler manifolds with their curvature property. A necessary condition for such a metric is that its total scalar curvature is positive, which is determined by its K\"ahler class:
\[
 \int_XS(\omega)\,\omega^n
 =2\pi n\,c_1(X)\cdot\alpha^{n-1}
 =-2\pi n\,K_X\cdot\alpha^{n-1}.
\]
Thus, the existence of a K\"ahler metric with positive scalar curvature implies the canonical bundle $K_X$ is not pseudoeffective. In the projective setting, the canonical bundle not being pseudoeffective is equivalent to uniruledness by the theorem of Boucksom--Demailly--P\u{a}un--Peternell \cite[Corollary~0.3]{BDPP13}. Based on this result, Heier and Wong \cite[Theorem~1.1]{HeierWong2012} showed that the existence of a K\"ahler metric with positive total scalar curvature implies uniruledness. More recently, Ou \cite[Theorem~1.1]{Ou2025} established the same equivalence for arbitrary compact K\"ahler manifolds. Consequently, every compact K\"ahler manifold admitting a K\"ahler metric with positive  scalar curvature is uniruled.

A natural converse question asks whether every uniruled compact K\"ahler manifold admits a K\"ahler metric with positive scalar curvature. For compact K\"ahler surfaces, the answer is affirmative by the work of Yau \cite{Yau1974} and Brown \cite{Brown2026}. Yau \cite{Yau1974} established that the existence of a K\"ahler metric with positive total scalar curvature is equivalent to its Kodaira dimension $\kappa(X)=-\infty$, which, for compact K\"ahler surface, is equivalent to the uniruledness. More recently, Brown \cite{Brown2026} proved every such surface admits a K\"ahler metric with positive scalar curvature. In higher dimensions, R\u{a}sdeaconu \cite{Rasdeaconu2005,Rasdeaconu2009} studied a related question suggested by LeBrun, whether the condition Kodaira dimension equals \(-\infty\) guarantees the existence of a K\"ahler metric of positive total scalar curvature. This question forms part of a broader effort to characterize uniruledness through the scalar curvature of K\"ahler metrics. 

In this direction, Yang explicitly formulated the following conjecture concerning positive total scalar curvature. A strengthened formulation involving pointwise positive scalar curvature appears in \cite[Conjecture~7.4]{Brown2026}.

\begin{conjecture}[Yang {\cite[Conjecture~4.7]{Yang2019}}]
\label{conj:yang}
A compact K\"ahler manifold \(X\) is uniruled if and only if it admits a K\"ahler metric of positive total scalar curvature.
\end{conjecture}

Our main result gives a negative answer to Conjecture \ref{conj:yang}.

\begin{theorem}\label{thm:A}
For every integer $n\geq3$, there exists a smooth projective manifold $X$ of complex dimension \(n\) such that
\begin{enumerate}[label=\textup{(\roman*)}]
  \item $K_X\notin\Psef(X)$, equivalently, $X$ is uniruled;
  \item $K_X\cdot\alpha^{n-1}>0$ for every $\alpha\in\Kah(X)$;
  \item $X$ admits a Riemannian metric of positive scalar curvature.
\end{enumerate}
\end{theorem}

By \textup{(ii)}, \(X\) admits no K\"ahler metric with nonnegative scalar curvature. Nevertheless, \textup{(iii)} ensures that the underlying smooth manifold admits a Riemannian metric with positive scalar curvature. Thus this distinction between K\"ahler and Riemannian positive scalar curvature persists even among uniruled projective manifolds.

\vspace{2mm}

\noindent \textbf{Statement on the use of AI.}
The explicit example in Appendix \ref{app:toric-seed} was constructed by the agent \href{https://github.com/frenzymath/Rethlas}{Rethlas} via ChatGPT-5.6-sol. All mathematical statements are written by the authors. 



\section{The admissible toric threefold and its abelian cover}

In this section, we construct the three-dimensional counterexample in Theorem \ref{thm:A} as an abelian cover of a smooth projective toric $3$-fold. We first specify the required properties of the base and construct the cover. We then compute its cohomology, identifying the additional divisor classes, and establish the canonical class inequality in Proposition \ref{prop:criterion}.

\subsection{Admissible toric $3$-folds} 

We introduce the class of admissible toric $3$-folds, which serve as base manifolds for the abelian cover. An explicit admissible toric $3$-fold is constructed in Appendix~\ref{app:toric-seed}.

\begin{definition}\label{def:seed}
A smooth projective toric $3$-fold $M^3$ is called \emph{admissible} if there exist a nef Cartier divisor $H$ and a torus-invariant prime divisor $D'\subset M^3$ satisfying the following conditions:  

\begin{enumerate}[label=\textup{(S\arabic*)}]
  \item  $\mathcal{O}_M(H)|_{D^\prime}\simeq \mathcal{O}_{D^\prime}$ and $ H^2\cdot D>0$ for any torus-invariant prime divisor $D\neq D'$. 
 
  \item The $\mathbb{Q}$-divisor  $L:=K_M+\frac83H $ is not pseudoeffective and satisfies that for any $\alpha\in \Kah(M)$
  \[
     L\cdot \alpha^2>0 \qquad \text{and }\qquad
     L\cdot {D'}^2\geq0.
  \]
  \item There are integral curves $C_1\subset D'$ and $C_2\subset M^3$  such that 
  \begin{equation}\label{interset-assump}
     D'\cdot C_1\neq0,\qquad
     c_1(M)\cdot C_1\equiv1\pmod2,\qquad
     H\cdot C_2\not\equiv0\pmod3. 
  \end{equation}
\end{enumerate}
\end{definition}

Throughout this section, we fix a triple $(M,H,D')$ satisfying
the conditions of Definition~\ref{def:seed}. Since $H$ is nef,  \(\mathcal{O}_M(H)\) is  globally generated by \cite[Theorem~6.3.12]{CLS11}. We may choose a nonzero torus eigensection  $s\in  H^0(M, \mathcal O_M(H))$. Its zero divisor is torus invariant and takes the form   
\begin{equation}\label{divisor-decom}
 \operatorname{div}(s)=\sum_\rho a_\rho D_\rho,
 \qquad a_\rho\geq0,
\end{equation}
where $\rho$ ranges over the rays of the fan of $M$ and \(D_\rho\) denotes the corresponding  torus-invariant prime divisors. 

\begin{remark}\label{positivity-H3} Since $\mathcal{O}_M(H)|_{D'}\simeq\mathcal{O}_{D'}$, its first Chern class vanishes. Thus, \[H\cdot D'\cdot D=(H|_{D'})\cdot (D|_{D'})=0\] for every divisor $D$. 

There exists $\rho$ such that $a_\rho>0$ and $D_\rho\neq D'$. Otherwise, $\operatorname{div}(s)=a D'$, where $a\ge 0$. It follows that 
$
 H^2\cdot D
 =H\cdot\operatorname{div}(s)\cdot D
 =aH\cdot D'\cdot D=0
$ for any torus-invariant prime divisor $D$
which leads to a contradiction with (S1) in Definition \ref{def:seed}.  Consequently, 
\[
H^3=H^2\cdot \operatorname{div}(s)=\sum_\rho a_\rho H^2\cdot D_\rho>0.
\]
Thus $H$ is big by \cite[Theorem~2.2.16]{Lazarsfeld04}.
\end{remark}

\subsection{Construction of the abelian cover} \label{sec:abliean-cover}

By Remark \ref{positivity-H3}, the complete linear system $|H|$ gives a morphism
\[
\varphi_H\colon M\to
\mathbb P\bigl(H^0(M,\mathcal O_M(H))^*\bigr)
\] whose image is 3-dimensional. Since $\cO_M (H)$ is globally generated and restricts trivially to $D'$, the members of $|H|$ disjoint from $D'$ form a nonempty Zariski open subset. Bertini's smoothness theorem \cite[Corollary~III.10.9]{Hartshorne77} and irreducibility theorem \cite[Theorem~3.3.1]{Lazarsfeld04} allow us to choose four distinct smooth irreducible divisors \(\Delta_1,\Delta_2,\Delta_3,\Delta_4\in|H|\) such that
\[
D'\cap\Delta_i=\varnothing \quad \text{for }1\leq i\leq4,
\qquad
\sum_{i=1}^4\Delta_i
\text{ has simple normal crossings}.
\]

For each $i$, choose a section $s_i\in H^0(M,\cO_M(H))$ with $\operatorname{div}(s_i)=\Delta_i$. Write $K:=\mathbb{C}(M)$ and consider the Kummer extension 
\begin{equation}
    K':=K\left(
 \sqrt[3]{s_1/s_4},\sqrt[3]{s_2/s_4},\sqrt[3]{s_3/s_4}\right).
\end{equation}
For $1\leq i, j\leq 3$, the valuation along $\Delta_i$ satisfies $\operatorname{ord}_{\Delta_i}(s_j/s_4)=\delta_{i j}$. Hence, $s_1/s_4,s_2/s_4,s_3/s_4$ are linearly independent in $K^*/ K^{*3}$. Kummer theory \cite[Theorem~5.30 and Remark~5.32]{MilneFT}  gives
\begin{equation}
    [K':K]=3^3 \qquad G:=\text{Gal}(K': K)\cong (\mathbb{Z}/3\mathbb{Z})^3
\end{equation}

To describe the Galois action explicitly, write $t_j:=\sqrt[3]{s_j/s_4}\in K'$ for $1\leq j\leq 3$, and fix a primitive cube root of unity $\zeta_3$. We may choose a basis $e_1$, $e_2$ and $e_3$ of $G$ such that 
\begin{equation}
    e_i(t_j)=\zeta_3^{\delta_{ij}}t_j
\end{equation}

\vspace{2mm}

Let $\pi: Z\rightarrow M$ be the normalization of $M$ in $K'$. Since $[K':K]=27$ , it is a finite morphism  of degree $27$. The Galois action on $K'$ preserves the integral closures of the affine coordinate rings of \(M\) in \(K'\) and induces a \(G\)-action on $Z$. The morphism $\pi$ identifies $M$ with the quotient $Z/G$.

\begin{remark}\label{inertia-group}
The branch locus of this cover is $\cup_{i=1}^4 \Delta_i$.
For $i=1,\dots,4$, let $G_i\subset G$ denote the inertia subgroup along $\Delta_i$, equivalently, the stabilizer of a geometric point lying over a general point of $\Delta_i$. Then
\[
G_i=\langle e_i\rangle
\quad \text{for }1\leq i\leq3,
\qquad
G_4=\langle -e_1-e_2-e_3\rangle.
\]
In particular, each $G_i$ is cyclic of order $3$.

Indeed, along $\Delta_i$ for $1\leq i\leq3$, only $s_i/s_4$ has nonzero valuation, namely $1$, so the inertia subgroup is generated by $e_i$. Along $\Delta_4$, all three ratios $s_j/s_4$ have valuation $-1$, and the inertia subgroup is generated by the automorphism
$
t_j\longmapsto\zeta_3^{-1}t_j
$
which corresponds to $-e_1-e_2-e_3$.
\end{remark}

Any point of $M$ lies on at most three branch components.
Suppose that $p\in M$ lies on exactly $r$ components, say $\Delta_{i_1},\ldots,\Delta_{i_r}$. Their inertia generators are linearly independent in $G$ (see Remark \ref{inertia-group}).  Thus , for every $z\in\pi^{-1}(p)$, the germ of $\pi$ at $z$ is analytically isomorphic to
\[
(\mathbb C^3,0)\longrightarrow(\mathbb C^3,0),
\qquad
(w_1,w_2,w_3)\longmapsto
(w_1^3,\ldots,w_r^3,w_{r+1},\ldots,w_3),
\]where the branch components through $p$ correspond to the coordinate hyperplanes $\{x_1=0\},\ldots,\{x_r=0\}$. 

In particular, the branch divisor has simple normal crossings, and the natural map
\[
G_{i_1}\times\cdots\times G_{i_r}\longrightarrow G
\]
is injective whenever the corresponding components meet. Using Pardini’s smoothness criterion \cite[\S1, p.~722]{Pardini98}, we conclude that $Z$ is smooth. Moreover, since $G$ is finite and $M$ is projective, the map $\pi: Z\rightarrow M$ is proper. Because \(\pi\) is finite and \(M\) is projective, the pullback of an ample line bundle on \(M\) is ample on \(Z\). Hence \(Z\) is projective.

\subsection{Eigensheaf and cohomology} The Galois action on $Z$ induces an action on $\pi_*\cO_Z$. Since $G$ is finite abelian and the base field is $\mathbb{C}$, this sheaf decomposes into eigensheaves indexed by the character group  $G^*:=\operatorname{Hom}(G, \mu_3)$, where $\mu_3\subset \mathbb{C}^*$ denotes the group of cube roots of unity.  
For any $\chi\in G^*$, define the eigensheaf by $(\pi_*\cO_Z)^\chi$
\begin{equation}
    (\pi_*\cO_Z)^\chi(U):=\{x\in \cO_Z(\pi^{-1} U)~|~g\cdot x=\chi(g)x \text{ for any } g\in G\}
\end{equation}
for any open set $U\subset M$. Thus,  
\begin{equation}\label{char-decomp}
   \pi_*\cO_Z=\bigoplus_{\chi\in G^*}(\pi_*\cO_Z)^{\chi} 
\end{equation}
The eigensheaf corresponding to the trivial character $\mathbf{1}\in G^*$ is $G$-invariant. Since $M=Z/G$, it follows that  $(\pi_*\cO_Z)^{\mathbf{1}}\simeq\cO_M$. 

\vspace{2mm}

Let $e_1^*,e_2^*, e^*_3$ be the dual characters with $e^{*}_{i}(e_j)=\zeta^{\delta_{ij}}_3\in \mu_3$. Each $\chi\in G^*$ can be written uniquely 
\begin{equation}
    \chi=(e^*_1)^{a_1(\chi)}\cdot(e^*_2)^{a_2(\chi)}\cdot(e^*_3)^{a_3(\chi)}, \quad a_i(\chi)\in\{0,1,2\}
\end{equation}
Define $a_4(\chi)\in \{0,1,2\}$ with $\chi(-e_1 -e_2-e_3)=\zeta^{a_4(\chi)}_3$. Thus, $a_4(\chi)\equiv-\sum_{i=1}^3a_i(\chi)\quad (\operatorname{mod} 3)$. It follows that 
\begin{equation}\label{constant-def}
    k_\chi
:=
\frac13\sum_{i=1}^4a_i(\chi)
=
\left\lceil\frac{\sum_{i=1}^3a_i(\chi)}3\right\rceil
\end{equation}
In particular, $k_{\mathbf{1}}=0$, whereas $k_\chi\in \{1,2 \}$ for $\chi\neq \mathbf{1}$. 

\begin{lemma}\label{eigen-isom-fun-sheaf}For any  $\chi\in G^*$, there is an isomorphism of $\cO_M$-modules 
\begin{equation}\label{eq:eigen-O}
(\pi_*\cO_Z)^\chi\simeq \cO_M(-k_\chi H).
\end{equation}
\end{lemma}
\begin{proof} Fix $\chi\in G^*$ and set $u_\chi:=t_1^{a_1(\chi)}\cdot t_2^{a_2(\chi)}\cdot t_3^{a_3(\chi)}\in K'$. By the definition of the Galois action, it follows that for any $g\in G$
\begin{equation}
    g\cdot u_\chi=\chi(g)u_\chi .
\end{equation}
Moreover, every non-zero element $u\in K'$ of character $\chi$ (i.e. $g\cdot u=\chi(g)u$ for any $g\in G$) is of the form $u=qu_\chi$, where $q\in K$. Indeed, $u/u_\chi$ is $G$-invariant and belongs to $(K')^G=K$.  

\vspace{2mm}

Let $U=\operatorname{Spec} R\subset M$ be any non-empty affine open subset and $R'$ be the integral closure of $R$ in $K'$. Thus, $\pi^{-1}(U)=\operatorname{Spec} R'$ and  
\[
\Gamma(U, (\pi_*\cO_Z)^\chi)=R'\cap K\cdot u_\chi. 
\]
Moreover,  $qu_\chi\in R'$ if and only if $(qu_\chi)^3\in R$. Indeed, if $qu_\chi\in R'$, then $(q u_\chi)^3$ belongs to $R'\cap K=R$, since $R$ is integrally closed. Conversely, if $(qu_\chi)^3\in R$, then $qu_\chi$ satisfies the monic polynomial $T^3-(qu_\chi)^3\in R[T]$, which implies that $qu_\chi\in R'$. 

\vspace{2mm}

For $q\neq 0$,  $(qu_\chi)^3$ is regular on $U$ if and only if its divisor on $U$ is effective.  Since 
\[\operatorname{div}_U ((u_\chi q)^3)=3\operatorname{div}_U(q)+\sum_{i=1}^3a_i(\chi)\Delta_i|_U-\sum_{i=1}^3a_{i}(\chi)\Delta_4|_U\] and $0\leq a_i(\chi)\le 2$,
it follows by \eqref{constant-def}
\[ 
(qu_\chi)^3\in R\iff \operatorname{div}_U(q)\geq k_\chi \Delta_4|_U\iff q\in \Gamma (U, \cO_M(-k_\chi\Delta_4)).
\]
Multiplication by $u_\chi$ gives an isomorphism 
\[
\cO_M(-k_\chi H)\simeq\cO_M(-k_\chi \Delta_4)\simeq(\pi_*\cO_Z)^\chi,
\]
which completes the proof.
\end{proof}

Combining the isomorphism \eqref{eq:eigen-O} with the character decomposition \eqref{char-decomp}  also obtains the cohomology-vanishing result on  $Z$. 
\begin{corollary}
The cohomology $H^{q}(Z, \cO_Z)$ vanishes for $q=1,2$. 
\end{corollary}
\begin{proof} Since $H$ is nef and big, it follows by Serre Duality and the Kawamata--Viehweg vanishing theorem \cite{Kawamata82,Viehweg82} that for $q=0,1,2$ and $k\in\mathbb{Z}_+$
\begin{equation}\label{vanish1}
    H^q(M, \cO_M(-k H))=0.
\end{equation}Moreover, Demazure vanishing theorem \cite{Demazure70} (see also \cite[Theorem~9.2.3]{CLS11}) yields that $H^{q}(M, \cO_M)=0$ for any $q>0$. Since $\pi$ is a finite morphism, it follows from these vanishings and $k_\chi\in \{1, 2\}$ for $\chi\neq \mathbf{1}$  that for $q=1, 2$
\begin{equation}\begin{split}
    H^{q}(Z, \cO_Z)\cong H^q(M, \pi_*\cO_Z)&\cong H^{q}(M, \cO_M)\bigoplus_{\chi\neq \mathbf{1}, \chi\in G^*} H^q(M, (\pi_*\cO_Z)^\chi)\\&\cong H^{q}(M, \cO_M)\bigoplus_{\chi\neq \mathbf{1}, \chi\in G^*} H^q(M, \cO_M(-k_\chi H))=0.
    \end{split}
\end{equation}
Here the three isomorphisms follow, respectively, from the finiteness
of $\pi$, the decomposition \eqref{char-decomp}, and
Lemma~\ref{eigen-isom-fun-sheaf}. The final equality follows from
\eqref{vanish1} and the vanishing of $H^q(M,\cO_M)$. \end{proof}

\subsection{Cohomological splitting and divisor classes}The Galois action on $Z$ also induces an action on $\pi_*\Omega^1_Z$, where $\Omega^1_Z$ is the sheaf of holomorphic $1$-forms on $Z$. Define the corresponding eigensheaf 
\begin{equation*}
(\pi_*\Omega^1_Z)^\chi(U):=\{x\in \Omega^1_Z(\pi^{-1}U)~|~g\cdot x=\chi(g)x \quad \text{ for any } g\in G\}. 
\end{equation*} for any open set $U\subset M$. 
Thus, 
\begin{equation}\label{char-decomp-1-form}
    \pi_*\Omega^1_Z=\bigoplus_{\chi\in G^*} (\pi_*\Omega^1_Z)^\chi.
\end{equation}
The eigensheaf $(\pi_*\Omega^1_Z)^\mathbf{1}$ is $G$-invariant and isomorphic to $\Omega^1_M$. 

\begin{proposition}For any non-trivial character $\chi\in G^*$, the cohomology group $H^{1}(M, (\pi_*\Omega^1_Z)^\chi)$ is isomorphic to $\mathbb{C}$. Moreover, 
\begin{equation}\label{cohom-split}
H^{1,1}(Z,\mathbb C)
\cong \pi^*H^{1,1}(M,\mathbb C)
\oplus\bigoplus_{\substack{\chi\in G^*\\\chi\ne\mathbf1}}
H^1\bigl(M,(\pi_*\Omega_Z^1)^\chi\bigr)
\cong \pi^*H^{1,1}(M,\mathbb C)\oplus\mathbb C^{26}.
\end{equation}
\end{proposition}
\begin{proof} 
Fix a non-trivial character $\chi\in G^*$. 
For each  $\Delta_i$, the twisted restriction sequence is 
\begin{equation*}
    0\rightarrow \cO_M(-\Delta_i-k_\chi H)\rightarrow \cO_M(-k_\chi H)\rightarrow \cO_{\Delta_i}(-k_\chi H)\rightarrow 0.
\end{equation*}Since $\Delta_i\sim H$, the first term is isomorphic to
$\cO_M(-(k_\chi+1)H)$. The associated long exact sequence in cohomology, together with \eqref{vanish1}, therefore gives that for $q=0,1 $,
\begin{equation}\label{vanish2}
    H^{q}(M, \cO_{\Delta_i}(-k_\chi H))=0.
\end{equation}

Let $\Delta_\chi:=\sum_{a_i(\chi)\neq 0}\Delta_i$. The Poincar\'e residue sequence  \cite{CLS11} for  \(\Delta_\chi\), twisted by $-k_\chi H$, is 
\begin{equation}\label{eq:residue}
0\to\Omega_M^1(-k_\chi H)\to
\Omega_M^1(\log\Delta_\chi)\otimes\cO_M(-k_\chi H)\to
\bigoplus_{\Delta_i\subset\Delta_\chi}\cO_{\Delta_i}(-k_\chi 
H)\to0.
\end{equation}
By \cite[Proposition~1.2]{Pardini98}, the middle term is isomorphic
to $(\pi_*\Omega_Z^1)^\chi$. Thus \eqref{vanish2} and the long
exact sequence associated with \eqref{eq:residue} yield
\begin{equation}
H^1\bigl(M,(\pi_*\Omega_Z^1)^\chi\bigr)
\cong H^1\bigl(M,\Omega_M^1(-k_\chi H)\bigr).
\end{equation}

It remains to show that $H^1(M, \Omega^1_M(-k_\chi H))\cong\mathbb{C}$. The toric Euler sequence \cite[Theorem~8.1.6]{CLS11}, twisted by $-k_\chi H$, is 
\begin{equation}\label{eq:toric-Euler-seq}
0 \rightarrow \Omega_M^1\left(-k_\chi H\right) \rightarrow \bigoplus_\rho \mathcal{O}_M\left(-D_\rho-k_\chi H\right) \rightarrow \operatorname{Pic}(M) \otimes \mathcal{O}_M\left(-k_\chi H\right) \rightarrow 0 .
\end{equation} 
where $D_\rho$ ranges over the torus-invariant prime divisors of $M$.
By \eqref{vanish1}, the last term has vanishing cohomology in degrees $0$ and $1$. The associated long exact sequence therefore yields 
\begin{equation} \label{eq:HD}
H^1(M,\Omega_M^1\left(-k_\chi H\right)) = \bigoplus_\rho H^1(M, \mathcal{O}_M\left(-D_\rho-k_\chi H\right) ).
\end{equation}

To complete the computation, we consider the restriction sequence for $D_\rho$, twisted by
$\cO_M(-k_\chi H)$
\begin{equation}\label{eq:Drho_seq}
0\longrightarrow\cO_M(-D_\rho-k_\chi H)
\longrightarrow\cO_M(-k_\chi H)
\longrightarrow\cO_{D_\rho}(-k_\chi H)
\longrightarrow0.
\end{equation}
By \eqref{vanish1}, the middle term has vanishing cohomology in
degrees $0$ and $1$. The associated long exact sequence  gives
\begin{equation}\label{rho-cohom}
H^1\bigl(M,\cO_M(-D_\rho-k_\chi H)\bigr)
\cong H^0\bigl(D_\rho,\cO_{D_\rho}(-k_\chi H)\bigr).
\end{equation}

If $D_\rho=D'$, it follows from \textup{(S1)} in
Definition~\ref{def:seed} that
$\cO_{D'}(-k_\chi H)\simeq\cO_{D'}$.  Hence
\[
H^0\bigl(D',\cO_{D'}(-k_\chi H)\bigr)\cong\mathbb C.
\]
If $D_\rho\ne D'$, then $H|_{D_\rho}$ is nef and has positive
self-intersection by \textup{(S1)} in Definition \ref{def:seed}. Thus 
\[
H^0\bigl(D_\rho,\cO_{D_\rho}(-k_\chi H)\bigr)=0
\qquad\text{for }D_\rho\ne D'.
\]
Otherwise, there is a nonzero section \(s\) of $\cO_{D_\rho}(-k_\chi H)$. The divisor $E:=\{s=0\}$ is  effective and $E\sim-k_\chi H|_{D_\rho}$, giving $0\leq H|_{D_\rho}\cdot E=-k_\chi H^2\cdot D_\rho<0$, 
a contradiction.

\vspace{2mm}

Combining these observations with \eqref{eq:HD} and
\eqref{rho-cohom}, we conclude that
\begin{equation}\label{eigen-sheaf-cohomology-C}
    H^1\bigl(M,(\pi_*\Omega_Z^1)^\chi\bigr)\cong H^1\bigl(M,\Omega_M^1(-k_\chi H)\bigr)\cong\mathbb C.
\end{equation}
The decomposition \eqref{cohom-split} now follows from
\eqref{char-decomp-1-form} and \eqref{eigen-sheaf-cohomology-C}.
\end{proof}

Since  $D'$ is  a complete toric variety, it is simply-connected \cite[Theorem~12.1.10]{CLS11}. Moreover, $D'$ is disjoint from the branch locus $\cup_{i=1}^4 \Delta_i$ and the restriction $\pi^{-1}(D')\rightarrow D'$ is finite \'etale of degree $|G|$. Thus, the inverse image $\pi^{-1}(D')$ has $|G|$-components, labeled by the transitive $G$-action 
\begin{equation}\label{eq:split}
\pi^{-1}(D')=\coprod_{g\in G}D'_g,
\qquad
\pi|_{D'_g}\colon D'_g\xrightarrow{\sim}D'.
\end{equation}
The divisor classes  $\{[D'_g]\}_{g\in G}$ are linearly independent in $H^{1,1}(Z, \mathbb{C})$. Indeed, let $C_1\subset D'$ be the curve appearing in (S3) of Definition \ref{def:seed} and let $C^g_1\subset D'_g$ be its inverse image under $\pi|_{D'_g}$. The disjointness of components $D'_g$, together with the projection formula, gives  
\begin{equation}
    C_{1}^g\cdot D'_{g'}=0 \text{ for } g\neq g'\qquad
    C_{1}^g\cdot D'_{g'}=C_1^g\cdot \pi^*D'=C_1\cdot D'\neq 0 \text{ for } g=g'
\end{equation}
where we have used \eqref{interset-assump}  in the latter case. This establishes the linear independence of $\{[D'_g]\}_{g\in G}$. 

\vspace{2mm}

For any $\chi\neq \mathbf{1}$, the following element $w_\chi$ spans the space $H^{1}(M, (\pi_*\Omega^1_Z)^\chi)$ under the natural cohomological identification 
\begin{equation}\label{eign-generator}
    w_\chi:=\sum_{g\in G} \chi(g)[D'_g]\in H^{1,1}(Z, \mathbb{C}) \quad \text{ with } g \cdot w_\chi=\chi(g)w_\chi \text{ for any } g\in G. 
\end{equation} 
The linear independence of the classes $[D_g']$ implies that $w_\chi\neq 0$ in $H^{1,1}(Z, \mathbb{C})$. 
Moreover, if $\chi\neq \mathbf{1}$, the sum $\sum_{g\in G} \chi(g)=0$ and all elements $\{w_\chi\}_{\chi\neq \mathbf{1}, \chi\in G^*}$ form a $26$-dimensional subspace
\begin{equation}\label{Def:W-space}
W:=\left\{
\sum_{g\in G}u_g[D'_g]:
u_g\in\mathbb C,\quad \sum_{g\in G}u_g=0
\right\}\subset H^{1,1}(Z,\mathbb C).
\end{equation}

\begin{corollary}\label{direct-sum} The intersection $W\cap \pi^*H^{1,1}(M, \mathbb{C})$ is a zero-dimensional space. Moreover, 
\begin{equation}\label{direct-sum-cohom}
    H^{1,1}(Z, \mathbb{C})=\pi^{*}H^{1,1}(M, \mathbb{C})\oplus W. 
\end{equation}
\end{corollary}
\begin{proof} Any element in $\pi^* H^{1,1}(M, \mathbb{C})$ is $G$-invariant while $W$ has a unique $G$-invariant element $0$. Thus, these two spaces intersects trivially. 

The isomorphism in \eqref{cohom-split} provides  $h^{1,1}(Z)=h^{1,1}(M)+26$. Since $\dim (W)=26$, these two spaces exhaust $H^{1,1}(Z, \mathbb{C})$, giving the decomposition \eqref{direct-sum-cohom}.
\end{proof}

\subsection{Positivity of the canonical class}
We now combine the ramification formula with the cohomological splitting above to study the canonical class of $Z$. The admissibility conditions on $M$ yield the following numerical properties.

\begin{proposition}\label{prop:criterion}
Let $M^3$ be an admissible toric $3$-fold. Under the above setting and notation, the   following identities hold 
\begin{equation}\label{eq:threefold-conclusions}
     K_Z\notin\Psef(Z),\qquad
 K_Z\cdot\alpha^2>0\quad\text{for every }\alpha\in\Kah(Z).
\end{equation}\end{proposition}

\begin{proof} Recall that the ramification index along each $\Delta_i$ equals $3$ for each $i$. Thus, using the Riemann--Hurwitz formula implies that 
\begin{equation}\label{eq:Kpull}
 K_Z \sim_{\mathbb Q} \pi^*\left(K_M+\frac23\sum_{i=1}^4\Delta_i\right)
     \sim_{\mathbb Q}\pi^*\left(K_M+\frac83H\right) \sim_{\mathbb Q} \pi^*L.
\end{equation}

By Corollary~\ref{direct-sum}, every $\alpha\in\Kah(Z)$ has a
unique expression
\begin{equation}\label{eq:alpha}
\alpha=\pi^*\beta+\sum_{g\in G}u_g[D'_g],
\qquad
\beta\in H^{1,1}(M,\mathbb R),\quad
u_g\in\mathbb R,\quad
\sum_{g\in G}u_g=0.
\end{equation}
Since $G$ permutes the components $D'_g$ transitively, averaging
over $G$ gives
$
\frac1{27}\sum_{h\in G}h^*\alpha=\pi^*\beta.
$
The K\"ahler cone is convex and preserved by $G$, so
$\pi^*\beta$ is K\"ahler.

We next show that $\beta\in\Kah(M)$. For every torus-invariant
curve $T\subset M$, choose an irreducible component
$\widetilde T\subset\pi^{-1}(T)$ mapping onto $T$, and let $d$
be its degree over $T$. The projection formula gives
\[
\beta\cdot T=\frac1d\,\pi^*\beta\cdot\widetilde T>0.
\]
Since $M$ is smooth, projective, and toric, its cone of curves
is generated by the finitely many invariant curves. These
inequalities therefore imply that $\beta$ is ample, and hence
K\"ahler \cite[Theorems~6.3.20 and~6.3.22]{CLS11}.

The components $D'_g$ are pairwise disjoint, and
\eqref{eq:split} gives
$
(D'_g)^2=\pi^*D'\cdot D'_g.
$
Thus \eqref{eq:Kpull}, \eqref{eq:alpha}, and the projection
formula yield
\begin{align*}
K_Z\cdot\alpha^2
&=27L\cdot\beta^2
  +2\left(\sum_{g\in G}u_g\right)L\cdot D'\cdot\beta
  +\left(\sum_{g\in G}u_g^2\right)L\cdot(D')^2\\
&=27L\cdot\beta^2
  +\left(\sum_{g\in G}u_g^2\right)L\cdot(D')^2
>0,
\end{align*}
where the final inequality follows from $\beta\in\Kah(M)$
and condition \textup{(S2)} in Definition~\ref{def:seed}.

Finally, if $K_Z$ is pseudoeffective, then its pushforward
$\pi_*K_Z\equiv27L$ would also be pseudoeffective, contradicting
\textup{(S2)} of Definition \ref{def:seed}.
\end{proof}

\section{Proof of Theorem \ref{thm:A}}
\label{sec:proof-main-theorem}
In Section~\ref{sec:abliean-cover}, we constructed a $3$-fold $Z$ that provides a counterexample to Conjecture~\ref{conj:yang}. By Proposition~\ref{prop:criterion}, $Z$ admits no K\"ahler metric with positive total scalar curvature. In this section, we show that $Z$ is simply connected and admits a complete Riemannian metric with positive scalar curvature.

\begin{proposition}\label{prop:topology}
 Let $M$ be an admissible toric $3$-fold and let $H$ and $D'$ be defined as in Definition \ref{def:seed}. Suppose that the abelian cover $\pi: Z^3\rightarrow M$ is constructed as in Section \ref{sec:abliean-cover}. Then $Z$ is simply-connected and non-spin. Moreover, $Z$ admits a Riemannian metric with positive scalar curvature.  
\end{proposition}

Remark that since $M$ is a smooth complete toric variety, it is simply
connected by \cite[Theorem~12.1.10]{CLS11}.

\begin{proof}

Recall that $Z^3$ and $M^3$ are projective and the map $\pi: Z^3\rightarrow M^3$ is a finite morphism of degree $27$. Moreover, its branch locus is a simple normal crossing divisor $\sum_{i=1}^4\Delta_i$ and for each $i\neq j$, $\Delta_i\cdot \Delta_j\cdot H=H^3>0$. Thus, their intersection has codimension two and each $\Delta_i$ is flexible in the sense of \cite[Definition~1.4]{Catanese84}. 

Let $G_i\subset G$ be the inertia subgroup along $\Delta_i$. By Remark \ref{inertia-group}, these groups are generated by $e_1$, $e_2$, $e_3$ and $e_4:=-e_1-e_2-e_3$, respectively. Identify each $G_i$ with $\mathbb{Z}_3$ sending $1$ to $e_i$, and define  
\[
\phi\colon\bigoplus_{i=1}^4G_i\longrightarrow G,
\qquad
(b_1,b_2,b_3,b_4)\longmapsto
\sum_{i=1}^3b_ie_i-b_4\sum^3_{i=1} e_i.
\]
Since $e_1$, $e_2$ and $e_3$ form a basis of $G$, we have that $\ker\phi=\langle(1,1,1,1)\rangle
\cong\mathbb Z_3$.

For each $i$, define the group homomorphism $\psi_i\colon H_2(M,\mathbb Z)\longrightarrow G_i$
\begin{equation*}
    \psi_i(x):=
\bigl\langle c_1(\mathcal O_M(\Delta_i)),x\bigr\rangle
\bmod3, \qquad \psi:=(\psi_1, \cdots, \psi_4): H_2(M, \mathbb{Z})\rightarrow \bigoplus_{i=1}^4 G_i.
\end{equation*}
 Since $\Delta_i\sim H$ for each $i$, it follows that 
$\psi(x)=(\langle c_1(\mathcal O_M(H)),x\bigr\rangle)\cdot (1,1,1,1)$. 
Thus, $\operatorname{Im} \psi\subset \ker \phi$. By (S3) in Definition \ref{def:seed}, the curve $C_2$ satisfies   $H\cdot C_2\not\equiv0\pmod3$.
Hence $\psi([C_2])$ is a generator of the cyclic group 
$\ker\phi$. Consequently, 
\begin{equation}
    \operatorname{Im}\psi=\ker\phi .
\end{equation}
Catanese's formula \cite[Proposition~1.8]{Catanese84} now gives
\[
\pi_1(Z)\cong
\ker\phi/\operatorname{Im}\psi=0.
\]
Thus $Z$ is simply connected.

\vspace{2mm}

It remains to show that $Z$ is non-spin. Let $C_1\subset D'$ be the curve in Definition~\ref{def:seed} (S3), and fix a lift $C_1^g\subset D'_g$ under the isomorphism in \eqref{eq:split}. Since $D'$ is disjoint from the branch locus, $\pi$ is \'etale in a neighborhood of $\pi^{-1}(D')$,  which induces the isomorphism  
\[
TZ|_{C_1^g}\simeq (\pi|_{C_1^g})^*(TM|_{C_1}).
\]
Combining it with the isomorphism $\pi|_{C^g_1}: C^g_1\rightarrow C_1$ and (S3) in Definition \ref{def:seed} yields
\[
  c_1(TZ)\cdot C^g_1=c_1(TM)\cdot C_1\equiv1\pmod2.
\]
The second Stiefel--Whitney class of $T_{\mathbb{R}}Z$ is the mod-two reduction of \(c_1(TZ)\). Hence,
\[
  \left\langle w_2(T_{\mathbb R}Z),[C_1^g]\right\rangle
  =c_1(TZ)\cdot C_1^g\equiv 1\pmod2.
\]
Thus, $Z$ is non-spin.

The underlying smooth manifold of \(Z\) is closed, non-spin, simply-connected and of real dimension six. By \cite[Corollary~C]{GromovLawson80}, it admits a Riemannian metric of positive scalar curvature.
\end{proof}

\begin{proof}[Proof of Theorem \ref{thm:A}]
The triple $(M,H,D')$ constructed in Appendix~\ref{app:toric-seed} satisfies the conditions of Definition~\ref{def:seed}. Let $Z^3$ be the associated smooth projective $3$-fold constructed in Section~\ref{sec:abliean-cover}. It follows by Proposition~\ref{prop:criterion} that 
\[
K_Z\notin\Psef(Z),
\qquad
K_Z\cdot\alpha^2>0
\quad\text{for every }\alpha\in\Kah(Z).
\]
Moreover, Proposition~\ref{prop:topology} shows that $Z$ admits a Riemannian metric of positive scalar curvature. This gives the proof for $3$-dimensional case. 

\vspace{2mm}

For $n=3+m\geq 4$, fix an elliptic curve $E$ and let \[T:=E^m,\qquad X:=Z\times T.\]
Then, $X$ is a smooth projective variety of dimension $n$. Since  $T$ has trivial canonical bundle, the product formula gives  $K_X=p^*_Z K_Z$, where $p_Z: X^n\rightarrow Z^3$ denotes the projection map. 

\vspace{2mm}

We now verify the required numerical properties of $K_X$.  Since \(H^1(Z,\mathbb C)=0\), using the K\"unneth formula implies that 
\[
 H^{1,1}(X,\mathbb R)
 =
 p_Z^*H^{1,1}(Z,\mathbb R)
 \oplus
 p_T^*H^{1,1}(T,\mathbb R).
\]
Thus, any \(\Omega\in\Kah(X)\) has a unique expression
$
 \Omega=p_Z^*\alpha+p_T^*\beta.
$
Restricting $\Omega$ to the fibers of the two projections shows that $\alpha\in\Kah(Z)$ and $\beta\in\Kah(T)$. Since $K_X\sim p_Z^*K_Z$, the binomial expansion yields
\begin{equation}\label{eq:product-intersection}
 K_{X}\cdot\Omega^{n-1}
 =
 \binom{m+2}{2}
 (K_Z\cdot\alpha^2)
 (\beta^m)
 >0.
\end{equation}

We also have $K_X\notin\Psef(X)$. Indeed, suppose that $c_1(K_X)$ were represented by a closed positive $(1,1)$-current $\Theta$. Choose a K\"ahler form $\omega_T$ on $T$ and $V_T:=\int_T\omega_T^m>0$. Then,
\[
\Theta'
:=
\frac1{V_T}(p_Z)_*
\bigl(\Theta\wedge p_T^*(\omega_T^m)\bigr)
\]
is a closed positive $(1,1)$-current on $Z$. The projection formula gives
\[
[\Theta']
=
\frac1{V_T}(p_Z)_*
\bigl(p_Z^*c_1(K_Z)\smile p_T^*[\omega_T]^m\bigr)
=
c_1(K_Z),
\]
contradicting $K_Z\notin\Psef(Z)$.

\vspace{2mm}

Since $Z$ admits a Riemannian metric with positive scalar curvature and $T$ is flat, $X$ also has a Riemannian metric with positive scalar curvature. This completes the proof.
\end{proof}

\appendix

\section{An explicit admissible toric $3$-fold}\label{app:toric-seed}

We construct an explicit triple $(M,H,D')$ satisfying Definition~\ref{def:seed}.

\subsection{The toric $3$-fold} 
Consider the polytope
\begin{equation}\label{eq:polytope}
 \mathcal P=\left\{(x,y,z)\in\mathbb R_{\geq0}^3:
 x\leq2,\ y\leq x+2,\ y+z\leq2x+1\right\}.
\end{equation}
Its primitive inward face normal vectors are
\begin{equation}\label{eq:rays}
\begin{aligned}
 v_0&=(1,0,0),& v_1&=(1,-1,0),& v_2&=(0,1,0),\\
 v_3&=(0,0,1),& v_4&=(2,-1,-1),& v_5&=(-1,0,0).
\end{aligned}
\end{equation}
Writing $ijk$ for $\operatorname{Cone}(v_i,v_j,v_k)$, the maximal cones of the normal fan $\Sigma$ of $\mathcal P$ are 
\begin{equation}\label{eq:maxcones}
 235,\ 245,\ 135,\ 145,\ 023,\ 024,\ 034,\ 134.
\end{equation}
The determinants of their primitive ray generator matrices, in this order, are
\[
 -1,\ 1,\ 1,\ -1,\ 1,\ -1,\ 1,\ -1.
\]
Thus $\Sigma$ is smooth. Since $\mathcal{P}$ is   a full-dimensional lattice polytope, its normal fan defines a smooth projective toric $3$-fold $M:=X_\Sigma$ 
\cite[Theorem~3.1.19(a) \& Theorem~6.2.1]{CLS11}.

\subsection{Divisor classes and cones} Let $D_i$ be the invariant prime divisor corresponding to $v_i$ and define
\[
 D':=D_0,\qquad P:=[D_2],\qquad Q:=[D_1+D_5],\qquad H:=[D_5].
\]
In what follows, we write $D_i$ and $H$ for their divisor classes. The divisor sequence \cite[Theorem 4.1.3 \& Proposition 4.2.6]{CLS11} gives 
\begin{equation}\label{eq:boundary-classes}
\begin{array}{c|c|c|c|c|c|c}
 \text{divisor}&D_0&D_1&D_2&D_3&D_4&D_5\\ \midrule
\text{class}&-2P+Q&Q-H&P&P-Q+H&P-Q+H&H.
\end{array}
\end{equation}
Consequently,
\begin{equation}\label{eq:canonical-class}
 \operatorname{Pic}(M)=\mathbb ZP\oplus\mathbb ZQ\oplus\mathbb ZH,
 \qquad K_M=-P-2H.
\end{equation}
Every effective divisor class on a complete toric variety has an invariant effective representative \cite[Proposition~4.3.2]{CLS11}. Moreover, $D_2=D_1+D_3$,
$D_4=D_3$ and 
$D_5=D_0+D_1+2D_3$. It follows that 
\begin{equation}\label{eq:psef}
 \operatorname{Psef}(M)=\operatorname{Cone}(D_0,D_1,D_3).
\end{equation}
These three generators are linearly independent. 

The wall-intersection formula \cite[Proposition~6.4.4]{CLS11} gives the distinct numerical classes of invariant curves, expressed in the coordinates $(P\cdot C,Q\cdot C,H\cdot C)$:
\begin{equation}\label{eq:curve-classes}
 (0,1,1),\ (1,2,2),\ (1,2,1),\ (0,0,1),\ (1,0,0),\ (0,1,0).
\end{equation}
These vectors have non-negative coordinates and include the coordinate vectors. Since invariant curves generate the Mori cone, duality gives
\begin{equation}\label{eq:kahler-cone}
 \overline{\operatorname{NE}}(M)=\mathbb R_{\geq0}^3,
 \qquad
 \operatorname{Kah}(M)
 =\{xP+yQ+zH:x,y,z>0\};
\end{equation}
see \cite[Theorem~6.3.20 and ~6.3.22]{CLS11}.

\subsection{Intersection numbers}The fan gives the relations
\begin{equation}
    D_0D_1=D_0D_5=D_1D_2=0, \quad D_2D_3D_4=D_3D_4D_5=0
\end{equation} in Chow ring. Together \eqref{eq:boundary-classes} and the normalization $\deg(D_2D_3D_5)=1$, these determine the required intersection number  \cite[Theorem~12.4.4 \& (12.5.9)]{CLS11}, 
\begin{equation}\label{eq:needed-intersections}
 H^2\cdot(D_0,D_1,D_2,D_3,D_4,D_5)=(0,1,2,1,1,3).
\end{equation}
Let \(L=K_M+\frac83 H\) and $B=xP+yQ+zH$. It follows that 
\begin{equation}\label{eq:L-intersections}
 L\cdot B^2
 =\frac23\left(x^2+xy+xz+y^2+2yz\right),
 \qquad
 L\cdot D'^2=2.
\end{equation}

\subsection{Admissible properties} 
We now verify Definition \ref{def:seed}. It follows by \eqref{eq:kahler-cone} that  
\begin{center}
$H$ is nef. 
\end{center}
No cone of $\Sigma$ contains both $\R_{\ge0} v_0$ and $\R_{\ge0}v_5$. Hence,  the orbit--cone correspondence \cite[Theorem 3.2.6]{CLS11} gives $D_0\cap D_5=\varnothing$. The defining section of $D_5$ restricts to a nowhere-vanishing section on $D'=D_0$, which implies 
\[
\cO_M(H)|_{D'}\simeq \cO_{D'}.
\] 
Together with \eqref{eq:needed-intersections}, this proves (S1) in Definition \ref{def:seed} .

Using \eqref{eq:boundary-classes} and \eqref{eq:canonical-class} gives $L=-P+\frac23H=\frac23D_0-\frac13D_1+\frac13D_3$. Since the generators in \eqref{eq:psef} are linearly independent, the negative coefficient of $D_1$ implies  $L\notin\operatorname{Psef}(M)$. For any  $B=xP+yQ+zH\in\operatorname{Kah}(M)$, the coefficients $x$, $y$ and $z$ are positive.  Using  \eqref{eq:L-intersections} yields
\[
 L\cdot B^2>0,
 \qquad
 L\cdot D'^2=2\geq0.
\]
which verifies (S2) in Definition \ref{def:seed}. 

\vspace{2mm}

Write  $V(\tau)=\overline{O(\tau)}$ for the orbit closure with a cone $\tau\in \Sigma$ and consider the integral curves
\[
 C_1=V(\operatorname{Cone}(v_0,v_2))\subset D_0=D',
 \qquad
 C_2=V(\operatorname{Cone}(v_1,v_3)).
\]
The wall-intersection formula gives
\[
 (P\cdot C_1,Q\cdot C_1,H\cdot C_1)=(1,0,0),
 \qquad
 (P\cdot C_2,Q\cdot C_2,H\cdot C_2)=(0,0,1).
\]
Since $D'=-2P+Q$ and $c_1(M)=P+2H$, it follows that
\[
 D'\cdot C_1=-2\neq0,
 \qquad
 c_1(M)\cdot C_1\equiv1\pmod2,
 \qquad
 H\cdot C_2\not\equiv0\pmod3.
\]
These identities verify (S3) in Definition \ref{def:seed}. Therefore, $(M,H,D')$ satisfies all three admissibility conditions in Definition \ref{def:seed}.

\vspace{2mm}

For reproducibility, an exact arithmetic verification of all the fan, cone, and intersection calculations used above is checked by the Python script available at \url{https://github.com/Ricciflow19/toric-3fold}. Running \texttt{python3 verify\_toric\_3fold.py} terminates successfully with \texttt{All exact checks passed.}

\bibliographystyle{alpha}
\bibliography{references}

\bigskip

\noindent
\textsc{Zehao Sha}\\
Institute for Mathematics and Fundamental Physics, Hefei, 230088, China\\
\textit{Email address:} \texttt{zhsha@imfp.org.cn}

\bigskip

\noindent
\textsc{Jian Wang}\\
State Key Laboratory of Mathematical Sciences,
Academy of Mathematics and Systems Science, Chinese Academy of Sciences,
Beijing 100190, China\\
\textit{Email address:} \texttt{jian.wang.4@amss.ac.cn}

\end{document}